\documentclass[11pt,reqno]{amsart}

\usepackage{amsmath}
\usepackage{amssymb}
\usepackage{amscd}
\usepackage{centernot} 
\usepackage{enumerate}
\usepackage{mathabx}
\usepackage{eucal}
\usepackage{epic}
\usepackage{graphicx}
\usepackage{wrapfig}
\usepackage[all]{xy}
\usepackage{hyperref}
\usepackage{color}

\theoremstyle{plain}
\newtheorem{theorem}{Theorem}[section]
\newtheorem{lemma}[theorem]{Lemma}

\theoremstyle{remark}
\newtheorem{remark}[theorem]{Remark}       

\newtheorem{example}[theorem]{Example}

\newtheorem{assumption}[theorem]{Assumption}

\newtheorem*{definition*}{Definition}

\newcommand{\hilit}[1]{{{{#1}}}}  

\newcommand{\pirt}{\varPi^h_{\scriptscriptstyle{\mathrm{RT}}}}
\newcommand{\prj}{\varPi^h_{k-1}}
\newcommand{\Poincare}{Poincar{\'{e}}}     
\newcommand{\Babuska}{Babu{\v{s}}ka}
\def\d{\partial}
\newcommand{\ip}[1]{\langle {#1} \rangle}
\newcommand{\veps}{\varepsilon}
\newcommand{\Tfae}{The following are equivalent}
\newcommand{\CCC}{\mathbb{C}}
\newcommand{\RRR}{\mathbb{R}}
\newcommand{\DD}{\mathcal{D}}
\newcommand{\diam}{\mathop\mathrm{diam}}
\newcommand{\dive}{\mathop\mathrm{div}}
\newcommand{\grad}{\mathop{\mathrm{grad}}}
\newcommand{\Hdiv}[1]{H(\dive,{#1})}

\newcommand{\hB}{\hat B}

\newcommand{\hb}{\hat b}
\newcommand{\hx}{\hat x}
\newcommand{\hz}{\hat z}
\newcommand{\hw}{\hat w}
\newcommand{\hX}{\hat X}
\newcommand{\hq}{\hat q}
\newcommand{\hqh}{\hat q_{n,h}}
\newcommand{\hr}{\hat r}
\newcommand{\hrh}{\hat r_{n,h}}

\newcommand{\om}{\varOmega}
\newcommand{\oh}{{\varOmega_h}}

\newcommand{\Xoh}{X_{h,0}}
\newcommand{\xoh}{x_{h,0}}

\newcommand{\Yohr}{Y_{h,0}^r}

\newcommand{\xh}{x_h}

\def\Grad{\nabla}

\title[DPG convergence rates]
{Convergence rates of the DPG method with \\ reduced test space degree}

\author[Bouma]{Timaeus Bouma}
\author[Gopalakrishnan]{Jay Gopalakrishnan}
\author[Harb]{Ammar Harb}

\begin{document}

\begin{abstract}  
This paper presents a duality theorem of the Aubin-Nitsche type for
  discontinuous Petrov Galerkin (DPG) methods. This explains the
  numerically observed higher convergence rates in weaker norms.
  Considering the specific example of the mild-weak (or primal) DPG
  method for the Laplace equation, two further results are obtained.
  First, the DPG method continues to be solvable even when the test space
  degree is reduced, provided it is odd. Second, a non-conforming
  method of analysis is developed to explain the numerically observed
  convergence rates for a test space of reduced degree.


\end{abstract}


\address{\hrule\bigskip Address for all authors: PO Box 751, Portland State University, Portland, OR 97207-0751.}

\thanks{This work was partially supported by the NSF under grant 
DMS-1318916 and by the AFOSR under grant FA9550-12-1-0484.}

\maketitle

\section{Introduction}  \label{sec:introduction}

The purpose of this note is to provide a theoretical explanation for
some numerically observed convergence rates of the discontinuous
Petrov-Galerkin (DPG) method. While some aspects of the theory that
follows are general, we will use the Laplace equation throughout as
the example to illustrate the main points.  There are two DPG methods
for the Laplace's equation. One is based on an ultra-weak
formulation~\cite{DemkoGopal11a} (where constitutive and conservation
equations are both integrated by parts) while the other is based on
the so-called mild-weak, or primal formulation, developed
in~\cite{BroerSteve12,DemkoGopal13a} (where only the conservation
equation is integrated by parts). The example which motivates our
study is the latter.

The method will be precisely introduced later. But to outline this
study, consider applying the method on a two-dimensional domain $\om$
meshed by a geometrically conforming finite element mesh of triangles
of mesh size $h$. The method produces an approximation $u_h$ to the
solution $u$ of the Laplace's equation in the interior of the mesh
elements, as well as an approximation to the flux $q$ on the element
interfaces. The first is a polynomial of degree at most $k_u$ on each
mesh element and the second is a polynomial of degree at most $k_q$ on
each mesh edge. The method uses test functions $v$ that are
polynomials of degree at most $k_v$ on each mesh element. It is the
interplay between the convergence rates and the degrees $k_u, k_q,k_v$
that we intend to study.

We identify three cases for study. Let $k\ge 1$ be an integer. The
cases are as shown:
\begin{center}
\begin{tabular}{r|ccc}
          &  $k_u$ &  $k_q$  &  $k_v$ 
  \\ \hline
  Case~1: &   $k$  &  $k-1$  & $k+1,$ 
  \\
  Case~2: &   $k-1$& $k-1$   & $k,$ 
  \\
  Case~3: &   $k$  & $k-1$   & $k.$ 
\end{tabular}  
\end{center}
The first case is the standard DPG setting for which error estimates
in the energy norm are proven in~\cite{DemkoGopal13a}. The other two
cases are motivated by a desire to reduce the test space degree and
have not been analyzed previously.

\hilit{What is the practical importance of reduced order test spaces? We
  give a three-part answer: First, consider the left hand side matrix
  of the linear system arising from the DPG method. Its assembly
  requires computation of the Gram matrix of the test space. Even
  though this matrix is block diagonal, it is of some practical
  interest to reduce the block size, especially when operating near
  the limit of memory bandwidth in multi-core architectures. Second,
  consider the right hand side computation. In cases where load terms
  are expensive to evaluate, reduction of test space degree brings
  significant computational savings. Finally, the third and the most
  compelling reason that prompted us to investigate this issue, is
  that there are practical limits on the degree of polynomials one can
  use in most finite element software. We prefer to hit this practical
  limiting degree with the trial space, rather than with the test space,
  because it is the approximation properties of the trial space that
  determines the final solution quality.  }

Our numerical experience with a few examples with smooth solutions,
one of which is fully reported in Section~\ref{sec:numerical}, is
summarized in Table~\ref{tab:rates}.  We observed that Case~2 is not
always stable: It yielded singular stiffness matrices for some even
$k$. However, when $k$ is odd, it converged, albeit at one order less
than the standard DPG case displayed in the first row. Keeping $k$ odd
and moving to Case~3, we find that the original DPG convergence rates
can be recovered, in spite of using a smaller $k_v$. Finally, we
observed that the convergence rate in $L^2(\om)$, in all cases, is one
order higher than in $H^1(\om)$.  These observations motivate our
ensuing theoretical studies.

\begin{table}[h]
\begin{center}
  \caption{Summary of numerically observed convergence rates}
\begin{tabular}{|l|c|c|}
\hline
          & \multicolumn{2}{c|}{$h$-convergence rates of $u_h$} \\
          \cline{2-3}
          &  in $H^1(\om)$ &  in $L^2(\om)$  
  \\ \hline
  Case~1   &   $k$          &  $k+1$     
  \\
  Case~2 
  ($k$ odd) &   $k-1$        & $k$
  \\
  Case~3 
  ($k$ odd) &   $k$  & $k+1$      
  \\ \hline
\end{tabular}  
\end{center}
  \label{tab:rates}
\end{table}
 
We explain the higher convergence rate in $L^2(\om)$ by developing a
duality argument for DPG methods. The duality theory is general and
can be applied beyond the Laplace example. We also give a complete
theoretical explanation for the even-odd behavior, including a
negative result by counterexample for even $k$, and a proof of a
positive result for odd $k$.  In explaining Case~3, we highlight a
connection between the DPG method and a weakly conforming method, and
show how to use a nonconforming-type analysis, using the second Strang
lemma, in the DPG context.

In the next section, we gather a number of abstract results applicable
to any DPG method in a general framework consisting of a trial space
of interior and interface variables. In Section~\ref{sec:laplace},
we introduce the DPG method for the Dirichlet problem and in distinct
subsections, provide explanations for the convergence rates in the
above-mentioned three cases. Finally in
Section~\ref{sec:numerical}, we present details of numerical
experiments \hilit{and discuss the practical importance of lower test order test spaces}.

\section{General results}  \label{sec:general}

Suppose $X_0$, $\hX$, and $Y$ are Hilbert spaces over $\CCC$.
Solutions are sought in the ``trial space'' $X=X_0 \times \hX$ and
have an ``interior'' component in $X_0$ and an ``interface'' component
in $\hX$.  Suppose there are continuous sesquilinear forms $\hat
b(\cdot,\cdot): \hX \times Y \to \CCC$ and $b_0(\cdot,\cdot): X_0
\times Y \to \CCC$, and let $b(\cdot,\cdot): X \times Y \to \CCC$ be
set by
\[
b( \,(w,\hw), y\,) = b_0( w, y) + \hb(\hw, y),
\]
for all $(w,\hw)\in X$ and $y\in Y$.  Let $Y^*$ denote the space of
continuous conjugate-linear functionals on $Y$. Given any $\ell \in Y^*$ we are interested in approximating an  $x \equiv (x_0, \hx) \in X$
satisfying
\begin{equation}
  \label{eq:weakform}
  b(x,y) = \ell(y)\qquad\forall y\in Y. 
\end{equation}
Let $\Xoh \subseteq X_0$ and $\hX_h \subseteq \hX$ be
finite-dimensional subspaces and let $X_h = \Xoh \times \hX_h$.  Let
$Y^r$ denote a finite-dimensional subspace of $Y$ and let $T^r: X \to
Y^r$ be defined by $ (T^r w, y)_Y = b(w,y)$ for all $y \in Y^r$. Here
and throughout $(\cdot,\cdot)_Y$ denotes the inner product in $Y$. The
DPG method for~\eqref{eq:weakform} computes $x_h \equiv (\xoh,\hx_h)$
in $X_h $ satisfying
\begin{equation}
  \label{eq:pdpg}
  b({\xh},{y}) = \ell({y}),
  \qquad\forall {y} 
  \in Y_h^r  =  T^r(X_h).
\end{equation}
A fundamental quasioptimality result for DPG methods is stated in
Theorem~\ref{thm:inex} below. It holds under these assumptions.

\begin{assumption}
  \label{asm:wellposed}
  Suppose $\{ z \in X: \;b(z, y)=0,\; \forall y\in Y\} = \{0\}$ and
  suppose there exist $C_1, C_2>0$ such that
  \begin{equation}
    \label{eq:9}
      C_1 \| y \|_Y \le 
  \sup_{0\ne z \in X} 
  \frac{| b(z,y)|}{ \| z \|_X } \le C_2 \| y \|_Y\qquad \forall y\in Y.
  \end{equation}
\end{assumption}

\begin{assumption}
  \label{asm:Pi}
   There is a linear operator $\varPi: Y \to Y^r$ and a
  $C_\varPi>0$ such that for all $ {w_h}\in {X_h}$ and all $ v \in Y$,
  \[
    b(  w_h , v - {\varPi}  v) =0, \qquad \text{and} \qquad
    \| {\varPi}  v \|_Y \le C_{{\varPi}} \|  v \|_Y.
  \]
\end{assumption}

\begin{theorem}[see~\cite{GopalQiu13}]
  \label{thm:inex}
  Suppose Assumptions~\ref{asm:wellposed} and~\ref{asm:Pi} hold.  Then
  the DPG method~\eqref{eq:pdpg} is uniquely solvable for $x_h$ and
  \[
  \| x- x_h \|_X \le \frac{C_2C_\varPi}{C_1}
  \inf_{z_h\in X_h } \| x - z_h \|_X
  \]
  where $x$ is the unique exact solution of~\eqref{eq:weakform}.
\end{theorem}

Another well-known result, motivated by~\cite{DahmeHuangSchwa11a}, is
an equivalence of the DPG method with a mixed Bubnov-Galerkin
formulation. To state it, we first define the error representation
function: let $\veps^r$ be the unique element of $Y^r$ satisfying
\begin{equation}
  \label{eq:epsr}
(\veps^r,y)_Y = \ell(y) - b(x_h,y),
 \forall y \in Y^r.
\end{equation}

\begin{theorem}
  \label{thm:pdggmixed}
  \Tfae\ statements:
  \begin{enumerate}[\quad i)\;]
  \item \label{item:dpgmixed-1} $x_h \in X_h$ solves the DPG
    method~\eqref{eq:pdpg}.
  \item \label{item:dpgmixed-2} $x_h\in X_h$ and $\veps^r\in Y^r$
    solve the mixed formulation
    \begin{subequations}
      \label{eq:mixeddpg}
      \begin{align}
        \label{eq:22}
        (\veps^r,y)_Y + b(x_h, y) & = \ell( y) &&\forall y\in Y^r, \\
        \label{eq:23}
        b(z_h,\veps^r)            & = 0      &&\forall z_h \in X_h.
      \end{align}
    \end{subequations}
  \end{enumerate}
\end{theorem}

\noindent Its simple proof is omitted (see e.g.~\cite{Gopal13}).
\begin{remark}
  \label{rem:controltesterror}
  The norm of $\veps^r$ is bounded by the error: Choosing $y=\veps^r$
  in~\eqref{eq:epsr}, we obtain
      \begin{align*}
        \|\veps^r\|^2_Y 
        & = (\veps^r,\veps^r)_Y  
          = \ell( \veps^r) - b(x_h, \veps^r)  
          = b(x -x_h, \veps^r). 
      \end{align*}
Hence, by Assumption~\ref{asm:wellposed}, 
\begin{equation}
  \label{eq:vepsbdd}
  \|\veps^r \|_Y\le C_2 \| x - x_h\|_X.
\end{equation}
This theme is further developed in~\cite{CarstDemkoGopal13}, where
$\|\veps^r\|_Y$ is established to be both a reliable and an efficient error
estimator.
\end{remark}

\subsection{Weakly conforming test space}

Let 
\begin{equation}
  \label{eq:30}
  Y_0^r = \{ y\in Y^r : \; \hb(\hw_h, y)=0,\;\; \forall \hw_h\,
  \in \hat X_h\}
\end{equation}
and let $T_0^r: X_0\to Y_0^r$ be defined by $(T_0^r w, y)_Y =
b_0(w,y)$ for all $y\in Y_0^r.$ In the examples we have in mind, $Y^r$
is a discontinuous Galerkin (DG) space, and $Y_0^r$ is a subspace with
weak interelement continuity constraints, i.e., a weakly conforming
space. In such cases, the application of the operator $T_0^r$ requires
a global inversion. We then compare these two DPG methods:
\begin{subequations}
\begin{align}
  \label{eq:26}
  \text{Find }(\xoh, \hx_h)\in X_h:
  &\quad
  b(\,(\xoh, \hx_h), y\,) = \ell(y) &&  \forall y\in Y_h^r \equiv T^r(X_h).
  \\
  \label{eq:27}
  \text{Find } \xoh\in \Xoh:
  &\quad
  b_0(\xoh, y) = \ell(y) &&  \forall y\in \Yohr \equiv T_0^r(\Xoh).
\end{align}
\end{subequations}
The first is the same as~\eqref{eq:pdpg}, the standard DPG method. We
view~\eqref{eq:26} as a ``hybridized'' form of the second
method~\eqref{eq:27}, and the next theorem shows in what sense they
are equivalent. The method~\eqref{eq:27} is not the preferred for
implementation due to the expense of applying $T_0^r$, but we will use
it later for error analysis.

\begin{theorem}
  \label{thm:pdpghybrid}
  The test spaces satisfy $\Yohr \subset Y_h^r$. Hence, if
  $(\xoh, \hx_h) \in X_h$ solves~\eqref{eq:26}, then $\xoh$
  solves~\eqref{eq:27}.
\end{theorem}
\begin{proof}
  Let $Y_\perp^r$ be the $Y$-orthogonal complement of $Y_h^r$ in
  $Y^r$. Then we have the orthogonal decomposition
  \begin{equation}
    \label{eq:17}
    Y^r = Y_h^r + Y_\perp^r.
  \end{equation}
  Let $y_0\in \Yohr$. Apply~\eqref{eq:17} to decompose $y_0 =
  y_h + y_\perp$, with $y_h\in Y_h^r$ and $y_\perp \in Y_\perp^r$.

  First, we claim that $y_\perp\in Y_0^r$. This is because
  \begin{align*}
    \hb(\hw_h,y_\perp)=(T^r(0,\hw_h),y_\perp)_Y =0\qquad\forall \hw_h\in \hX_h.
  \end{align*}
  The last identity followed from the orthogonality of $y_\perp$ to
  $T^r(X_h)$.

  Next, we claim that $y_\perp=0$. It suffices to prove that
  $(y_0,y_\perp)_Y=0$ since $  (y_0,y_\perp)_Y=
  \| y_\perp\|_Y^2$. Since $y_0\in \Yohr$, there is a $w_h\in \Xoh$
  such that $y_0=T^r_0w_h$. Then, 
  \begin{align*}
    (y_0,y_\perp)_Y
    & = (T^r_0 w_h, y_\perp)_Y  = b_0(w_h,y_\perp)
    && \text{as }y_\perp \in Y_0^r
    \\
    &=(T^r(w_h,0),y_\perp)_Y =0 
    && \text{as }T^r(X_h)  \perp y_\perp.
  \end{align*}
  Finally, since $y_\perp=0$, we have $y_0 = y_h+0\in Y_h^r.$ Thus
  $\Yohr \subset Y_h^r$. The second statement of the theorem is now
  obvious by choosing $y\in \Yohr$ in~\eqref{eq:26}.
\end{proof}

\subsection{Injectivity}

Let $B_h:X_h \to (Y^r)^*$ be the operator generated by the form
$b(\cdot,\cdot)$, i.e., 
\[
(B_h w_h)(y)= b( w_h,y),
\qquad \forall  w_h \in X_h,\; y \in Y^r.
\]
Similarly, let $\hB_h: \hX_h \to (Y^r)^*$ be defined by 
\begin{equation}
  \label{eq:Bh}
(\hB_h \hz_h)(y)
= \hb( \hz_h,y),
\qquad
\forall  \hz_h \in \hX_h, \; y\in Y^r.  
\end{equation}
The injectivity of $B_h$ yields the unique solvability of the DPG method.

\begin{assumption}
  \label{asm:inj}
  Suppose
  \begin{enumerate}[\quad a)\;]
  \item \label{item:injA1} 
     $\Xoh \subseteq Y^r$, 
  \item \label{item:injA2}
    $\hb( \hz_h, z_0)=0$ for all $\hz_h \in \hX_h$ and 
    $z_0 \in \Xoh$, and 
  \item \label{item:injA3} any $z_0 \in \Xoh$ satisfying $
    b_0(z_0,z_0)=0$ must be zero.
  \end{enumerate}
\end{assumption}

\begin{theorem} \label{thm:inj} 
  If $B_h$ is injective, then $\hB_h$ is injective, and the DPG
  method~\eqref{eq:pdpg} is uniquely solvable.  Conversely, if $\hB_h$
  is injective, then $B_h$ is injective, provided
  Assumption~\ref{asm:inj} holds.
\end{theorem}
\begin{proof}
  Suppose $B_h$ is injective. The injectivity of $\hB_h$ is obvious
  from $\hB_h\hw_h = B_{h}(0,\hw_h)$.  We also claim that $T^r$ is
  injective: Indeed, if $w_h \in X_h$ satisfies $T^r w_h=0$, then
  $0=(T^r w_h, y)_Y = b(w_h, y)=(B_h w_h)(y)$ for all $y\in Y^r$, so
  $w_h =0$. The injectivity of $T^r$ implies that $\dim(Y_h^r) = \dim
  (X_h)$, so the DPG method~\eqref{eq:pdpg} yields a square
  system. Moreover, since~\eqref{eq:pdpg} is the same as
  \[
  (T^r x_h, T^r w_h)_Y = \ell( T^r w_h)\qquad \forall \, w_h \in X_h,
  \]
  the injectivity of $T^r$ also implies that there is a unique solution
  $x_h$ in $X_h$.
  
  Now suppose $\hB_h$ is injective. To prove that $B_h$ is injective,
  consider a $(w_0,\hw) \in X_h$ satisfying $B_h(w_0,\hat{w})=0$. Then
  \begin{align*}
    0
    & =(B_h(w_0,\hat{w}))(w_0)
    && \text{ by Assumption~\ref{asm:inj}(\ref{item:injA1})}
    \\
    & =  b\left((w_0,\hat{w}),w_0 \right)
    =b_0(w_0,w_0)+\hat{b}(\hat{w},w_0)
    \\
    & =b_0(w_0,w_0),
    && \text{ by Assumption~\ref{asm:inj}(\ref{item:injA2})}.
  \end{align*}
  Therefore, by Assumption~\ref{asm:inj}(\ref{item:injA3}),
  $w_0=0$. It only remains to show that $\hw=0$. But $(\hB_h \hw)(y) =
  \hb(\hw,y)= b(\,(0,\hw), y) = (B_h(w_0, \hw))(y) =0$ for all $y \in
  Y^r$. Hence the injectivity of $\hB_h$ implies $\hw=0$.
\end{proof}

\subsection{Duality argument for DPG}

By virtue of \autoref{thm:pdggmixed}, we may rewrite the DPG
method~\eqref{eq:pdpg} as follows: Find $\xoh \in X_{0,h}, \; \hx_h
\in \hat X_h,$ and $\veps^r \in Y^r$ solving
\begin{subequations}
\label{eq:I}
  \begin{alignat}{4}
    \label{eq:2}
    b_0(w,\veps^r)          && & & & = 0      
    & & \qquad\forall w \in X_{0,h}, 
    \\
    & & \hb( \hw, \veps^r) & & & =0
    & & \qquad \forall \hw \in \hX_h,
    \\
    \label{eq:1}
    b_0(\xoh, y)  \,& +\, & \hb(\hx_h,y)  \,& +  \, 
    &   (\veps^r,y)_Y & = \ell( y),
    & & \qquad\forall y\in Y^r.
  \end{alignat}
\end{subequations}
Defining 
\[
a( z,\hz ,v | w,\hw , y) = 
 \overline{b_0(w,v)}
+\overline{\hb( \hw, v)}
+ b_0(z, y) +
\hb(\hz,y)+(v,y)_Y,
\]
the mixed system~\eqref{eq:I} can then be rewritten as 
\[
a( \xoh,\hx_h
, \veps^r| w,\hw , y) = \ell(y), \qquad \forall 
 w\in X_{0,h}, \hw \in
\hX_h, y\in Y^r,
\] 
\hilit{where the complex conjugate  on the first two terms make the form $a$ sesquilinear.}  Now, observe that with $\veps=0$, the exact solution $( x_0, \hx ,
\veps) \in X_0 \times \hX \times Y$ satisfies the same equation for
all $ w\in X_0, \hw \in \hX, y \in Y$. Hence, we have a `Galerkin
orthogonality' relation
\begin{equation}
  \label{eq:6}
   a( x_0-\xoh,\hx - \hx_h , \veps-\veps^r | w,\hw , y) = 0,
\end{equation}
for all $ w\in X_{0,h}, \hw \in \hX_h, y\in Y^r.$ Note also that 
\begin{align*}
|a( z,\hz ,v | w,\hw , y) |
& \le 
C_2\| (z,\hz)\|_X\|y\|_Y + C_2 \| (w,\hw)\|_X\|v\|_Y + \| v\|_Y \|y\|_Y
\\
& 
\le 
\left(C_2^2 \| (z,\hz)\|_X^2 + 2\| v\|_Y^2 \right)^{1/2}
\left(C_2^2 \| (w,\hw)\|_X^2 + 2\| y\|_Y^2 \right)^{1/2}
\\
& \le \| a \|\, \|(z,\hz,v)\|_{X_0 \times \hX \times Y}
     \|(w,\hw,y)\|_{X_0 \times \hX \times Y}
\end{align*}
where $\| a\|$ is a constant not larger than $\max(C_2^2,2)$. Under the
following assumption, we can extend the Aubin-Nitsche
technique~\cite{Nitsc68} to DPG methods, as seen in the next theorem.

\begin{assumption}  \label{asm:dual}
Suppose $L$ and $Z$ are Hilbert spaces such that the embeddings
$ Z \subseteq X_0 \times \hX \times Y$ and $X_0 \subseteq L$ are
continuous. 
Assume that there is a $C_3(h)>0$ such that
for any $g\in L$, there is a $U(g) \in Z$ satisfying
\begin{equation}
  \label{eq:5}
a(  w,\hw , y | U(g) ) = (w,g)_L  
\end{equation}
for all $(w,\hw , y) \in X_0 \times \hat X \times Y$ and 
\begin{equation}
  \label{eq:3}
  \inf_{ W \in X_{0,h} \times \hX_h \times Y^r} 
  \| U(g) - W \|_{ X_0 \times \hat X \times Y} \le C_3(h) \| g \|_L.
\end{equation}  
\end{assumption}

\begin{theorem}  \label{thm:duality}
Suppose  Assumption~\ref{asm:dual} holds. Then,
  \[
  \|x - \xoh \|_L \le C_3(h)
  \|a\| \| (x,\hx, \veps) - (\xoh,\hx_h, \veps^r) \|_{X_0 \times \hat X\times Y}.
  \]
\end{theorem}
\begin{proof}
  Setting $g=w=x-\xoh$, $\hw =
  \hx - \hx_h$, and $y=\veps-\veps^r$  in~\eqref{eq:5}, 
  \begin{align*}
    \| x -\xoh \|_L^2 
    & = a(x-\xoh, \hx-\hx_h, \veps-\veps^r | U(x-\xoh))
    \\
    & = a(x-\xoh, \hx-\hx_h, \veps-\veps^r | U(x-\xoh) - W), 
    \qquad \text{by~\eqref{eq:6},}
    \\
    & \le  \|a\| \| (x-\xoh, \hx-\hx_h, 
    \veps-\veps^r) \|_{X_0 \times \hat X \times Y}
    \| U(x-\xoh)- W\|_{X_0 \times \hat X \times Y}
  \end{align*}
  for any $W \in X_{0,h} \times \hX_h \times Y^r$. Hence~\eqref{eq:3}
  completes the proof.
\end{proof}

\begin{remark}
  \label{rem:unisoldual}
  Let $A : X_0 \times \hX \times Y \to (X_0 \times \hX \times Y)^*$ be
  the operator generated by $a(\cdot,\cdot)$, i.e.,
  $(A(z,\hz,v))(w,\hw,y) = a( z,\hz,v| w,\hw,y)$ for all $(z,\hz,v),
  (w,\hw,y) \in X_0 \times \hX \times Y.$ If
  Assumption~\ref{asm:wellposed} holds, then $A$ is a bijection. (This
  follows from the \Babuska-Brezzi theory~\cite{BrezzForti91}, applied
  to the mixed system~\eqref{eq:mixeddpg}: the ``inf-sup
  condition'' follows from~\eqref{eq:9}, and the ``coercivity in the
  kernel condition'' is trivial.) Hence, the dual operator of $A$ is
  also a bijection whereby we conclude that~\eqref{eq:5} has a unique
  solution $U(g)$.
\end{remark}

\begin{remark}
  All results of this section hold for spaces over the real field
  $\RRR$ -- one only needs to replace $\CCC$ by $\RRR$, sesquilinear
  by bilinear, and conjugate-linear by linear to obtain the
  corresponding statements for real valued function spaces.
  \hilit{The DPG method for the Helmholtz
    equation~\cite{GopalMugaOliva14} provides an example where
    sesquilinear forms over $\CCC$ are used. For simplicity, in the
    remaining sections we will restrict ourselves to real-valued
    functions.}
\end{remark}

\section{Application to the Laplace equation} \label{sec:laplace}

Suppose $\om$ is a bounded open polygon in $\mathbb{R}^{2}$ with
Lipschitz boundary, meshed by $\oh$, a geometrically
conforming shape regular finite element mesh of triangles. Let $h =
\max_{K\in\oh}\diam K$. Let $\d\oh$ denote the collection of all
element boundaries $\partial K$ for all elements $K$ in $\oh$. 
We now study the DPG
approximation to the Dirichlet problem
\begin{subequations}
  \label{eq:bvp}
  \begin{align}
  -\Delta u   & = f && \text{ on } \om, \\
  u & =0 && \text{ on } \d \om.
  \end{align}
\end{subequations}
All functions are real-valued in this section.

Omitting a detailed derivation of the method, which can be found
in~\cite{BroerSteve12,DemkoGopal13a}, we simply specify how the method
can be obtained by setting these within the general framework of
\autoref{sec:general}:
\begin{align*}
  X_0 & = H_0^1(\om), \qquad \hX = H^{-1/2}(\d\oh), 
  \\
  Y  & = H^1(\oh), \qquad \text{ where } 
  \\
  H^1(\oh)
  & = \{ v: \; v|_K \in H^1(K), \;\forall K \in \oh\},
  \\
  H^{-1/2}(\d \om_h)
  & = \{ \eta \in {\prod_K} H^{-1/2}(\d K)
  :\; \exists \,r
  \in \Hdiv\om \text{ such that } 
  \\
  & \qquad \eta|_{\d K} = r\cdot n|_{\d K}, 
  \quad \forall\, K \in \oh \},
\end{align*}
where $n$ denotes the unit outward normals on the boundary of mesh
elements.  The space $H^{-1/2}(\d \oh)$ is normed, as
in~\cite{RaviaThoma77a}, by
\begin{align}
  \label{eq:qnorm}
  \|  \hr_n \|_{H^{-1/2}(\d\oh)} 
  & =  \inf
      \big\{ \| r \|_{\Hdiv\om}:  \; r \in \Hdiv\om
      \text{ such that } \hr_n |_{\d K} = r\cdot n|_{\d K} \;\forall\,
      K\in \oh \big\}.
\end{align} 
The ``broken'' Sobolev space $H^1(\oh)$ is normed by
\begin{align}
\label{eq:vnorm}
\| v \|_{H^1(\oh)}^2
& = (v,v)_\oh + (\grad v,\grad v)_\oh.
\end{align} 
Throughout \hilit{the rest of the paper,} the derivatives are always calculated element by element,
and
\begin{align*}
  (r,s)_\oh =  \sum_{K\in\oh} (r,s)_K,
  \qquad
  \ip{\ell,w}_{\d\oh} = \sum_{K\in\oh} \ip{\ell,w}_{1/2,\d K},
\end{align*}
where $(\cdot,\cdot)_K$ denotes the $L^2(K)$-inner product and $\ip{
  \ell,\cdot}_{1/2,\d K}$ denotes the action of a functional $\ell$ in
$H^{-1/2} (\d K)$. The bilinear and linear forms of the weak
formulation are set by
\begin{align*}
  b_0(w,y) & = (\grad w, \grad y)_\oh,
  &
  \hb(\hr_n,y) & = -\ip{ \hr_n, y}_{\d\oh},
  & 
  \ell(y) 
  & = (f,y)_\om.
\end{align*}
Assumption~\ref{asm:wellposed} was verified for this formulation
in~\cite{DemkoGopal13a}. We will denote the exact solution of the
resulting weak formulation~\eqref{eq:weakform} by $(u,\hq_n) \in
X$. Note that $\hq_n|_{\d K} = \d_n u|_{\d K}$ for all $K\in \oh$.

To complete the specification of the method, it only remains to
set the discrete spaces. Let $P_k(D)$ denote the set of
polynomials of degree at most $k$ on the domain $D$ (with the
understanding that the set is trivial when $k<0$). Let $P_k(\oh) = \{ v:
v|_K \in P_k(K)$ for all $K\in \oh\}$ and let $P_k(\d\oh)$ denote the
set of functions $v$ on $\d\oh$ having the property $v|_E \in P_k(E)$
for all edges of $\d K$ and for all $K\in \oh$. Then, recalling the
three cases mentioned in \autoref{sec:introduction}, we set, for any integer
$k\ge 1$, 
\begin{alignat*}{3}
  &\text{Case~1}                       &&    \text{Case~2}                  && \text{Case~3}\\
   \Xoh & = P_k(\oh) \cap X_0          & \Xoh & = P_{k-1}(\oh) \cap X_0       & \Xoh & = P_k(\oh) \cap X_0, \\        
   \hX_h& = P_{k-1}(\d\oh)\cap\hX\qquad & \hX_h& = P_{k-1}(\d\oh)\cap\hX\qquad& \hX_h& = P_{k-1}(\d\oh) \cap \hX, \\   
   Y^r  & = P_{k+1}(\oh)                &  Y^r & = P_{k}(\oh)               & Y^r & = P_{k}(\oh).                         
\end{alignat*}
The discrete solution in each of these cases is denoted by $(u_h,
\hqh) \in X_h$.  We now proceed to study these cases and explain the
observations in \autoref{tab:rates}.

\subsection{Case~1: Application of the duality argument} \label{ssec:case1}

For Case~1, Assumption~\ref{asm:Pi} was verified
in~\cite{DemkoGopal13a}.  This then led to
\cite[Theorem~4.1]{DemkoGopal13a}, which states that
\[
  \| u - u_h \|_{H^1(\om)} + \| \hq_n - \hqh \|_{H^{-1/2}(\d\oh)}
  \le 
  C 
  \inf_{ (w_h,\hrh) \in X_h }
  \big(
  \| u - w_h \|_{H^1(\om)} + \| \hq_n - \hrh \|_{H^{-1/2}(\d\oh)}
  \big).
\]
Here and henceforth, $C$ denotes a generic constant independent of the
size of the triangles in $\oh$ (but dependent on mesh shape
regularity), whose value at different occurrences may vary.
As explained in previous papers (see e.g., \cite{DemkoGopal11a}), 
applications of the Bramble-Hilbert Lemma in the Lagrange and
Raviart-Thomas spaces show that
\begin{subequations}
  \label{eq:approxest}
\begin{align}
  \inf_{ w_h \in P_l(\oh) \cap X_0}
  \| u - w_h \|_{H^1(\om)} 
  &  \le C h^l |u |_{H^{l+1}(\om)},  && \forall l \ge 0,
  \\
  \inf_{ \hrh \in P_{m-1}(\d\oh) \cap \hX}
  \| \hq_n - \hrh \|_{H^{-1/2}(\d\oh)}
  & \le C h^m \left( | u |_{H^{m+1}(\om)} + | f |_{H^m(\om)} \right), 
  &&\forall m\ge 1.
\end{align}
\end{subequations}
Therefore, 
\begin{equation}
  \label{eq:energyest}
  \| u - u_h \|_{H^1(\om)} + \| \hq_n - \hqh \|_{H^{-1/2}(\d\oh)}
  \le 
  C h^k  \left( | u |_{H^{k+1}(\om)} + | f |_{H^k(\om)} \right).
\end{equation}
Hence the $O(h^k)$ convergence of $\| u - u_h \|_{H^1(\om)}$ (first
entry of \autoref{tab:rates}) is completely explained. To explain the
$O(h^{k+1})$ convergence of $\| u - u_h \|_{L^2(\om)}$, we apply the
duality argument of \autoref{thm:duality}. Its hypothesis is verified
in the next proof.

\begin{theorem}
  \label{thm:dual}
  Suppose $\om$ is convex. Then, for Case~1, 
  \[
  \| u - u_h \|_{L^2(\om)} \le 
  C h^{k+1} \left( | u |_{H^{k+1}(\om)} + | f |_{H^k(\om)} \right).
  \]
\end{theorem}
\begin{proof}
  Set 
\begin{align*}
  Z_1 &= H^2(\om) \cap X_0, & L & = L^2(\om),\\
  Z_2 &= H^2(\om) \cap Y,   &  Z & =Z_1 \times \hX \times Z_2.
\end{align*}
To verify Assumption~\ref{asm:dual}, let $g\in L$. By
Remark~\ref{rem:unisoldual}, there is a unique $U(g) \equiv (z,\hz_n,d)
\in X_0 \times \hX \times Y$ solving~\eqref{eq:5}. Writing
out~\eqref{eq:5} in component form,
\begin{subequations}
\label{eq:dual-laplace}
  \begin{alignat}{4}
    \label{eq:dual-laplace-1}
    (d,y)_Y  \, +  \, & (\grad z, \grad y)_\oh  \,-\, & \ip{\hz_n,y}_{\d\oh} 
    & = 0,
    &&\qquad\forall y\in Y, 
    \\
    \label{eq:dual-laplace-2}
    & (\grad d,\grad w)_\oh            & & = (g, w)_\oh      
    &&\qquad\forall w \in X_0,
    \\
    \label{eq:dual-laplace-3}
    &&\ip{ \hw_n, d}_{\d\oh} & =0
    &&\qquad \forall \hw_n \in \hX.
  \end{alignat}
\end{subequations}

We need to understand the regularity of solutions
of~\eqref{eq:dual-laplace}. Considering the  $d$ component first,
we claim that~\eqref{eq:dual-laplace-3}
implies $d \in H_0^1(\om)$: Indeed the distributional gradient
$\grad d$ acting on a test function $\phi \in \DD(\om)^2$
satisfies $ (\grad d)(\phi) = -(d, \dive \phi)_\oh = (\grad d,
\phi)_\oh - \ip{ d,\phi\cdot n}_{\d\oh} $ and the last term vanishes
by~\eqref{eq:dual-laplace-3}, so the distributional gradient is in
$L^2(\om)^2$. It is also easy to see that the trace of $d$ vanishes on
$\d\om$. Then, \eqref{eq:dual-laplace-2} implies that $-\Delta d = g$.
Next, consider $z \in H_0^1(\om)$. Equation~\eqref{eq:dual-laplace-1}
with $y\in H_0^1(\om)$ yields $ (\grad z, \grad y) = -(d,y)_\oh-(\grad
d, \grad y)_\oh = -(d,y)_\oh + (\Delta d, y)_\oh $ which implies
$\Delta z = d+g$. Finally, using the equations for $z$ and $d$ in
\eqref{eq:dual-laplace-1} and integrating by parts, we find $
\ip{\hz_n, y}_{\d\oh} = \ip{n\cdot \grad (d+z), y}_{\d\oh}.$
Summarizing, the classical form of~\eqref{eq:dual-laplace} is
\begin{subequations}
  \label{eq:7}
\begin{align}
  -\Delta d & = g,  && \text{ on } \om, 
  \\
  d & = 0,  && \text{ on } \d\om, 
  \\
  \Delta z & = d+g,   && \text{ on } \om, 
  \\
  z & = 0, && \text{ on } \d\om, 
  \\
  \hz_n & = n \cdot \grad(d+z),   && \text{ on } \d K, \; \forall K \in \oh.
\end{align}
\end{subequations}

Thus, by full regularity of the Dirichlet problem on a convex
domain~\cite{Grisv85}, $d$ and $z$ are in $H^2(\om)$, and moreover,
\begin{align*}
  \| d \|_{Z_2} & \le  C\| g \|_L, \\
  \| z \|_{Z_1} & \le C \left( \| d \|_L + \| g\|_L\right)\le  C\| g \|_L, \\
  \| \hz_n \|_{\hX} & \le \| \grad (d+z) \|_{\Hdiv\om} 
	\\
	& =\| \grad (d+z) \|_{L}+\|\Delta(d+z)\|_{L}
	\\
	& =\| \grad (d+z) \|_{L}+\|d\|_{L} \qquad \text{by~\eqref{eq:7},} 
	\\
	& \le C \|g \|_L.
\end{align*}
Hence
\begin{equation}
  \label{eq:4}
  \|(z,\hz,d)\|_Z \le C  \| g \|_L.
\end{equation}

To complete the verification of Assumption~\ref{asm:dual}, we now only
need to bound some approximation errors. By the Bramble-Hilbert lemma, 
\begin{align}
  \nonumber
  \lefteqn{\hspace{-.5cm}\inf_{ W \in X_{0,h} \times \hX_h\times Y^r }
  \| U(g) - W \|^2_{X_{0} \times \hX\times Y}}
 \\\label{eq:3terms}
  & =  
  \inf_{w_h \in P_k(\oh) \cap X_0} 
  \| z - w_h \|_{H^1(\om)}^2 +
  \inf_{v_h \in P_{k+1}(\d\oh)}
  \| d- v_h\|_{H^1(\oh)}^2  
  + 
  \inf_{\hw_h \in P_{k-1}(\d\oh)\cap \hX}
  \| \hz_n - \hw_h \|_{\hX}^2
  \\ \nonumber 
  & \le  C h^2\left(  | d |_{H^2(\om)}^2 + |z |_{H^2(\om)}^2 \right)
  +
  \inf_{r_h \in R_{k-1}}\| \grad(d+z) - r_h \|_{\Hdiv\om}^2
\end{align}
where $R_{k-1}$ is the Raviart-Thomas subspace~\cite{RaviaThoma77a} of
$\Hdiv\om$ consisting of all vector functions which when restricted to
an element takes the form $x p_1 + p_2$ for some $p_1 \in P_{k-1}(K)$ and
some $p_2 \in P_{k-1}(K)^2$. Let $\pirt$ denote the Raviart-Thomas
projection into $R_{k-1}$. By its well-known commutativity property
with the $L^2$-projection~$\prj$ onto $P_{k-1}(\oh)$, we have 
\begin{align*}
  \inf_{r_h \in R_{k-1}}\| \grad(d+z) - r_h \|_{\Hdiv\om}
& \le 
\| (I-\pirt)\grad(d+z)\|_{\Hdiv\om}
\\
& \le \| (I-\pirt)\grad(d+z)\|_L + 
\| (I- \prj)\Delta(d+z)\|_L
\\
& \le \| (I-\pirt)\Grad(d+z)\|_L + 
\| (I- \prj) d\|_L, \qquad \text{by~\eqref{eq:7},}
\\
& \le C h | d + z |_{H^2(\om)} + C h |d |_{H^1(\om)},
\end{align*}
where we used the Bramble-Hilbert lemma again in the final step.
Hence using the regularity estimate~\eqref{eq:4},
\[
    \inf_{ W \in X_{0,h} \times \hX_h \times Y^r} \| U(g) - W \|_{ X_0 \times \hat X \times Y}
  \le C h \| g \|_L,
\]
thus verifying Assumption~\ref{asm:dual}. Now, applying
Theorem~\ref{thm:duality}, 
\[
\| u - u_h \|_{L^2(\om)}
\le C h 
\left( 
  \| u - u_h \|_{H^1(\om)} + \| \hq_n - \hqh \|_{H^{-1/2}(\d\oh)}
  + \| \veps - \veps^r \|_{H^1(\oh)}
\right)
\]
where $\veps=0$ and $\veps^r$ is as in~\eqref{eq:epsr}. This implies,
by virtue of~\eqref{eq:vepsbdd} in Remark~\ref{rem:controltesterror},
\[
\| u - u_h \|_{L^2(\om)}
\le C h 
\left( 
  \| u - u_h \|_{H^1(\om)} + \| \hq_n - \hqh \|_{H^{-1/2}(\d\oh)}
\right)
\] 
so the proof is finished using~\eqref{eq:energyest}.
\end{proof}

\subsection{Case~2: Explaining the even-odd separation} \label{ssec:case2}

This case was not studied in previous works. We must first check if
the DPG system is solvable for this case. For this, \autoref{thm:inj}
is useful. Clearly, Assumption~\ref{asm:inj} holds -- in fact, it holds
for all the three cases: items (\ref{item:injA1}) and (\ref{item:injA2}) are
obvious, while (\ref{item:injA3}) follows by the \Poincare\
inequality. Hence, applying \autoref{thm:inj}, we conclude that the
DPG method in Case~2 is uniquely solvable if and only if $\hB_h$ is
injective.

\begin{example}\label{ex:example}
  We begin with a negative result showing that $\hB_h$ is not
  injective when $k=2$.  On a mesh consisting of a single element in
  the $xy$-plane, namely the unit triangle with vertices $a_0=(0,0),
  a_1=(1,0)$ and $a_2=(0,1),$ we choose a basis for $\hX_h$: Letting
  $e_i$ denote the edge opposite to $a_i$ and $1_{e_i}$ denote the
  indicator function of $e_i$, the basis is $(1_{e_2}, x|_{e_2},
  1_{e_1}, y|_{e_1}, 1_{e_0}/\sqrt 2, x|_{e_0} /\sqrt 2)$. For the
  trial space $Y^r$, we choose the polynomial basis ($1$, $x$, $y$,
  $x^2$, $xy$, $y^2$).  The stiffness matrix of the operator $\hB_h$
  with respect to these bases is
\[
\begin{pmatrix}
 1 & 1/2 & 1 & 1/2 & 1 & 1/2  \\
 1/2 & 1/3 & 0 & 0 & 1/2 & 1/3  \\
 0 & 0 & 1/2 & 1/3 & 1/2 & 1/6  \\
 1/3 & 1/4 & 0 & 0 & 1/3 & 1/4  \\
 0 & 0 & 0 & 0 & 1/6 & 1/12  \\
 0 & 0 & 1/3 & 1/4 & 1/3 & 1/12  
\end{pmatrix},
\]
whose determinant is zero. Hence, by theorem~\autoref{thm:inj} the DPG
method is not uniquely solvable in this example. 

\hilit{This example is closely related to a well-known
  result~\cite{NME:NME1620190405} that there is a nonzero quadratic
  function that is zero on the two Gauss-Legendre points (required for
  an exact integration of a third order polynomial) on each edge of a
  triangle. Clearly, such a quadratic function is  
  orthogonal to all functions that are linear on each edge of
  the triangle.  }
\end{example}

We now show that for odd $k$, the situation is better.

\begin{lemma}
  \label{lem:odd}
  Let $K$ be a triangle and $k\ge 1$ be an odd integer. Any $w$ in
  $P_k(K)$ satisfying
  \begin{subequations}
    \label{eq:odd-unique}
    \begin{align}
      \label{eq:odd-edges}
      \int_E w\, q\, ds & = 0 &&\forall\, q \in P_{k-1}(E), \;
      \forall \text{ edges } E \subset \partial K,
      \\ \label{eq:odd-int}
      \int_K w\, r\, dx& =0 &&\forall\,  r \in P_{k-3}(K), \text{ if } k\ge 3,
    \end{align}
  \end{subequations}
  must vanish on $K$.
\end{lemma}
\begin{proof}
  Equation~\eqref{eq:odd-edges} implies that $w|_E$ must be a scaled
  Legendre polynomial of degree exactly $k$ on $E$. Since $k$ is odd,
  this implies that the values of $w$ at the endpoints of each edge
  must have opposite signs. This is impossible unless $w$ vanishes on
  $\d K$.  But if $w|_{\d K}=0$, then $w\equiv 0$ if $k=1$. If $k\ge
  3$, then $ w=\lambda_1\lambda_2\lambda_3s_{k-3}$, for some
  $s_{k-3}\in P_{k-3}(K)$ where $\lambda_i$ is the $i$th barycentric
  coordinate.  Then~\eqref{eq:odd-int} implies $w\equiv 0$ on $K$.
\end{proof}

\begin{theorem}   \label{thm:case2}
  In Case~2, for odd $k\ge 3$, these statements hold:
  \begin{enumerate}[\quad i)\;]
  \item \label{item:case2:inj}
    The DPG method is uniquely solvable.
  \item  \label{item:case2:energyerror}
    The solution $(u_h, \hqh)$ of the DPG  method  satisfies 
    \begin{equation}
      \label{eq:energyest2}
      \| u - u_h \|_{H^1(\om)} + \| \hq_n - \hqh \|_{H^{-1/2}(\d\oh)}
      \le 
      C h^{k-1}  \left( | u |_{H^{k}(\om)} + | f |_{H^{k-1}(\om)} \right).
    \end{equation}
  \item  \label{item:case2:l2error} If $\om$ is convex, then 
    \begin{equation}
      \label{eq:l2estimate2}
      \| u - u_h \|_{L^2(\om)}
      \le 
      C h^{k}  \left( | u |_{H^{k}(\om)} + | f |_{H^{k-1}(\om)} \right).
    \end{equation}
  \end{enumerate}
\end{theorem}
\begin{proof}
  By \autoref{thm:inex}, if we verify Assumption~\ref{asm:Pi}, then
  the DPG method is uniquely solvable.

  To do so, we first claim that there exists a
  $C_\varPi>0$ and a unique $\varPi v \in P_{k}(K)$ for any $v\in
  H^1(K)$, such that
  \begin{subequations}
    \label{eq:Piodd}
    \begin{align}
      \label{eq:Piodd-edges}
      \int_E (v-\varPi v) q\, ds & = 0 &&\forall\, q \in P_{k-1}(E), \;
      \forall \text{ edges } E \subset \partial K,
      \\ \label{eq:Piodd-int}
      \int_K (v-\varPi v) r\, dx& =0 &&\forall\,  r \in P_{k-3}(K)
      \\ \label{eq:Piodd-bd}
      \|\varPi v\|_{H^1(K)} \leq C_\varPi &\|v\|_{H^1(K)} && \forall\, v \in H^1(K).
    \end{align}
  \end{subequations}
  It is easy to see that~\eqref{eq:Piodd-edges}--\eqref{eq:Piodd-int}
  forms a square system for $\varPi$, so existence of $\varPi v$
  follows from uniqueness. But uniqueness is already proved by
  Lemma~\ref{lem:odd}.  The estimate~\eqref{eq:Piodd-bd} follows from a
  simple scaling argument.

  The energy error estimate~\eqref{eq:energyest2} now follows from
  \autoref{thm:inex} and \eqref{eq:approxest}. The $L^2$ error
  estimate~\eqref{eq:l2estimate2} follows from \autoref{thm:duality}:
  The required verification of Assumption~\ref{asm:dual} proceeds as
  in the proof of \autoref{thm:dual} -- the only difference is in the
  degrees of approximation spaces in the first two infimums
  in~\eqref{eq:3terms}, a difference that is inconsequential for the
  rest of the arguments.
\end{proof}

Theorem~\ref{thm:case2} explains all entries in the second row of
\autoref{tab:rates}. The convergence rate in~\eqref{eq:energyest2} is
suboptimal and limited by the low degree of $u_h$. This motivates the
next case.

\subsection{Case~3: A nonconforming analysis} \label{ssec:case3}

The only difference between Case~2 and Case~3 is that the degree of
$u_h$ is increased by one. We analyze Case~3 using a technique of
analysis different from the previous subsection, appealing to
\autoref{thm:pdpghybrid} and the second Strang lemma (see
e.g.~\cite{Ciarl78}) in the analyses of nonconforming methods.

\begin{theorem}   \label{thm:case3}
  In Case~3, for odd $k\ge 1$, these statements hold:
  \begin{enumerate}[\quad i)\;]
  \item \label{item:case3:inj}
    $\hB_h$ is injective and the DPG method is uniquely solvable.
  \item  \label{item:case3:energyerror}
    The $u_h$-component of the solution satisfies 
    \begin{equation}
      \label{eq:case3:energyest2}
      \| u - u_h \|_{H^1(\om)} 
      \le 
      C h^{k}  \left( | u |_{H^{k+1}(\om)} + | f |_{H^{k}(\om)} \right).
    \end{equation}
  \item  \label{item:case3:l2error} If $\om$ is convex, then 
    \begin{equation}
      \label{eq:case3:l2estimate2}
      \| u - u_h \|_{L^2(\om)}
      \le 
      C h^{k+1}  \left( | u |_{H^{k+1}(\om)} + | f |_{H^{k}(\om)} \right).
    \end{equation}
  \end{enumerate}
\end{theorem}
\begin{proof}
  First, observe that if $k \ge 3$, then by the unisolvency of the DPG
  method in Case~2, namely \autoref{thm:case2}(\ref{item:case2:inj}),
  its $B_h$ is injective, which implies by \autoref{thm:inj} that
  $\hB_h$ of Case~2 is injective. But since 
  the flux ($\hat{X}_h$) and test spaces ($Y^r$)
  of Case~3 are identical to that of Case~2, both cases have the same
  $\hB_h$. Hence $\hB_h$ of Case~3 is injective and consequently by
  \autoref{thm:inj}, $B_h$ of Case~3 is injective. Thus we have proved
  the first statement of the theorem for $k\ge 3$. For $k=1$, if
  $(\hB_h \hrh)(w) = -\ip{ \hrh, w}_{\d\oh}=0$ for all $w \in Y^r$,
  then
  \[
  \int_{\d K} w\, \hrh\, ds  = 0, 
  \qquad \forall\, w  \in P_{k}(K).
  \]
  The matrix of this system (for $\hrh$) is the transpose of the matrix
  of~\eqref{eq:odd-unique} (for $w$), which is invertible by
  Lemma~\ref{lem:odd}. Hence $\hrh=0$, i.e., $\hB_h$ is injective when
  $k=1$.

  Next we prove~\eqref{eq:case3:energyest2}.  Recall that $Y_0^r$ is
  defined in~\eqref{eq:30} and $\Yohr$ in~\eqref{eq:27}.  By
  \autoref{thm:pdpghybrid}, $u_h \in \Xoh$ satisfies~\eqref{eq:27},
  i.e.,
  \begin{equation}
    \label{eq:discreteeq}
    b_0(u_h, y) = (f,y)_\om, \qquad \forall y\in \Yohr.    
  \end{equation}
  We proceed by viewing this as a nonconforming Petrov-Galerkin
  discretization of
  \[
    b_0(u,y) = (f,y)_\om, \qquad \forall y\in H_0^1(\om)
  \]
  and bounding the consistency error in an argument akin to the second
  Strang lemma.  Let $C_p$ denote the constant, derived from
  \Poincare\ inequality, such that $\| w \|_{H^1(\om)} \le C_p \|
  \grad w \|_{L^2(\om)}$ for all $w \in H_0^1(\om)$. Then, for any
  $w_h \in \Xoh$ 
  \begin{align}
    \nonumber
    \|u_h  -w_h\|_{H^1(\om)}
    & \le C_p
    \sup_{0\ne z_h\in \Xoh}\frac{(\grad (u_h-w_h),\grad z_h)_\om}
    {\|\grad z_h\|_{L^2(\om)}}
    \leq 
    C_p^2\sup_{0\ne z_h\in \Xoh}\frac{b_0(u_h-w_h, z_h)}
    {\|z_h\|_{H^1(\om)}}
    \\ \nonumber
    & \leq C_p^2\sup_{0\ne y\in Y^r_0}\frac{b_0(u_h-w_h,y)}
    {\|y\|_Y}
    =  C_p^2 \| T_0^r(u_h-w_h) \|_Y
    = C_p^2\sup_{0\ne y\in \Yohr}\frac{b_0(u_h-w_h,y)}{\|y\|_Y}
    \\ \nonumber
    & = 
     C_p^2\sup_{0\ne y\in \Yohr}\frac{b_0(u_h-u,y) + b_0(u-w_h,y)}{\|y\|_Y}
     \\ \label{eq:10}
     & = 
     C_p^2\sup_{0\ne y\in \Yohr}\frac{(f,y)_\om - b_0(u,y) + b_0(u-w_h,y)}{\|y\|_Y}, 
  \end{align}
  where we have used~\eqref{eq:discreteeq}. Since $b( (u,\hq_n), y) =
  (f,y)_\om$ for all $y \in Y$, the term representing the consistency
  error in~\eqref{eq:10} can be written as $ (f,y)_\om - b_0(u,y) =
  \hb(\hq_n,y)$. By the definition of $Y_0^r$ (see~\eqref{eq:30}), we also
  have $ \hb(\hq_n,y) =\hb(\hq_n-\hrh,y) $ for any $\hrh \in \hX_h$ and
  $y \in Y_0^r$. Therefore,
\[
\|u_h  -w_h\|_{H^1(\om)}
\le C_p^2 
\sup_{0\ne y\in \Yohr}\frac{ b( (u-w_h, \hq_n - \hrh),y)}{\|y\|_Y}
\le C_p^2 C_2 C
\left(\| \hq_n - \hrh \|_{\hX}+\| u - w_h \|_{H^1(\om)}\right).
\]
Since $\hrh$ and $\hq_n$ are element-by-element traces of an $r_h$ in $R_{k-1}$
and $q=\grad u$, respectively,
  \[
  \| \hrh - \hq_n \|_{\hX} \le \| r_h - \grad u  \|_{\Hdiv\om},
  \]
so
  \begin{align*}
    \|u_h  -w_h\|_{H^1(\om)}
    \le C 
    \left( 
      \inf_{r_h \in R_{k-1}} \| r_h - \grad u \|_{\Hdiv\om} + 
      \| u - w_h \|_{H^1(\om)} 
      \right).
  \end{align*}
  Finally, by the triangle inequality, 
  \begin{align*}
    \| u - u_h \|_{H^1(\om)} 
    & \le \| u - w_h \|_{H^1(\om)} + \| u_h - w_h \|_{H^1(\om)}
    \\
    & \le C \left( \| u - w_h \|_{H^1(\om)} + 
    h^k ( | u|_{H^{k+1}(\om)} + |f|_{H^k(\om)})\right)
  \end{align*}
  for any $w_h \in \Xoh$. Choosing $w_h$ to be an appropriate
  interpolant, the proof of \eqref{eq:case3:energyest2} is finished.

  The final estimate \eqref{eq:case3:l2estimate2} is proved by
  verifying Assumption~\ref{asm:dual} (along the lines of the proof of
  \autoref{thm:dual}) and applying \autoref{thm:duality}.
\end{proof}

The final row of \autoref{tab:rates} is now completely explained by
\autoref{thm:case3}.

\section{Numerical Results} \label{sec:numerical}

In this section, we report results from a numerical experiment. The
presented DPG method for the Laplace equation was used to solve the
Dirichlet problem with $\om$ set to the unit square. The function $f$
was chosen so that the exact solution is $u=sin(\pi x) sin(\pi y)$.  We
construct an $n \times n$ uniform mesh by dividing $\om$ into $n^2$
congruent squares and further subdividing each square into two
triangles by connecting the diagonal of positive slope.  Its mesh size
is $h = \sqrt 2/n$. The method is applied on a sequence of such meshes
with geometrically increasing $n$. The implementation of the method is
done using FEniCS~\cite{LoggMardalEtAl2012a,LoggWells2010a}. Computed discretization errors in Cases 1, 2, and 3
are reported.

A baseline is provided by Case~1, reported in \autoref{tab:case1}. The
last column reports the rate of convergence in $L^2(\om)$,
approximately calculated using two successive rows by $\log_2( \| u -
u_h \|_{L^2(\om)}/\| u - u_{h/2} \|_{L^2(\om)})$. The
$H^1(\om)$-convergence rate is computed similarly. We observe from the
table that the $L^2(\om)$-rate is one order higher than the
$H^1(\om)$-rate, as expected from Theorem~\ref{thm:dual}. 

\begin{table}
\centering\footnotesize{
\caption{Case 1: $(k_u,k_q,k_v) = (k,k-1,k+1)$} 
\begin{tabular}{|l|cc|cc|} 
\hline
$n$    & $\|u-u_h\|_{H^1(\om)}$& rate & $\| u - u_h \|_{L^2(\om)}$ &  rate 
\\ [0.5ex] 
\hline 
\multicolumn{5}{|c|}{$k=1$} \\
\hline
2     &1.53E+00&	0.86&	2.61E-01&	1.65 \\
4     &8.43E-01&	0.96&	8.33E-02&	1.90 \\
8     &4.32E-01&	0.99&	2.23E-02&	1.97\\
16     &2.18E-01&	1.00&	5.67E-03&	1.99\\
32     &1.09E-01&	1.00&	1.42E-03&	2.00 \\
64     &5.45E-02 &  &3.57E-04  & \\
\hline 
\multicolumn{5}{|c|}{$k=2$} \\
\hline
2     &4.67E-01&	1.85&	3.24E-02&	2.91\\
4     & 1.29E-01&	1.95&	4.31E-03&	2.98 \\
8     &3.34E-02&	1.99&	5.47E-04&	2.99 \\
16     &8.42E-03&	2.00&	6.87E-05&	3.00 \\
32     &2.11E-03& 2.00 &		8.60E-06&3.00 \\
64     &5.28E-04 &  &1.08E-06  & \\
\hline 
\multicolumn{5}{|c|}{$k=3$} \\
\hline
2     &1.01E-01&	2.94&	5.52E-03&	4.04 \\
4     &1.32E-02&	3.00&	3.36E-04&	4.07\\
8     &1.65E-03&	3.01&	2.00E-05&	4.04\\
16     &2.06E-04&	3.00&	1.22E-06&	4.02 \\
32     &2.57E-05&		&7.50E-08& \\
\hline
\end{tabular} 
\label{tab:case1} }
\end{table} 

Next, we consider Case~2, reported in \autoref{tab:case2}. The table
is computed similarly to Case~1, however only odd $k$ are considered
since the problem in Case~2 is not well posed for even $k$ -- see
Example~\ref{ex:example}. We observe that the $H^1(\om)$-convergence
is $O(h^{k-1})$, confirming the first theoretical estimate of
\autoref{thm:case2}. The rate of convergence is increased by one in
the next column in accordance with the second estimate of
\autoref{thm:case2}.

\begin{table}
\centering\footnotesize{
\caption{Case 2: $(k_u,k_q,k_v) = (k-1,k-1,k)$} 
\begin{tabular}{|l|cc|cc|} 
\hline
$n$    & $\|u-u_h\|_{H^1(\om)}$& rate & $\| u - u_h \|_{L^2(\om)}$ &  rate \\ [0.5ex] 
\hline 
\multicolumn{5}{|c|}{$k=3$} \\
\hline 
2     &4.67E-01&	1.85&	3.24E-02&	2.91 \\
4     &1.29E-01&	1.95&	4.31E-03&	2.98\\
8     &3.34E-02&	1.99&	5.47E-04&	2.99 \\
16     &8.42E-03&	2.00&	6.87E-05&	3.00\\
32     & 2.11E-03&2.00  &		8.60E-06&3.00 \\
64     &5.28E-04		 &  &1.08E-06  & \\
\hline 
\multicolumn{5}{|c|}{$k=5$} \\
\hline
2     &1.70E-02&	3.92&	7.24E-04&	4.90\\
4     &1.13E-03&	3.98&	2.43E-05&	4.97\\
8     &7.14E-05&	4.00&	7.76E-07&	4.99 \\
16     &4.48E-06& 4.00  &		2.44E-08&5.00 \\
32     &2.80E-07		 &  &7.64E-10  & \\
\hline
\end{tabular} }
\label{tab:case2} 
\end{table} 

Results from Case~3 are reported in \autoref{tab:case3}. We observe
that the $H^1(\om)$-convergence rate is $k+1$, the same as in Case~1,
even though the test space is of a lesser degree.  These observations
illustrate and confirm the theoretical results of \autoref{thm:case3}.

\begin{table}
\centering\footnotesize{
\caption{Case 3: $(k_u,k_q,k_v) = (k,k-1,k)$} 
\begin{tabular}{|l|cc|cc|} 
\hline
$n$    & $\|u-u_h\|_{H^1(\om)}$& rate & $\| u - u_h \|_{L^2(\om)}$ &  rate \\ [0.5ex] 
\hline
\multicolumn{5}{|c|}{$k=1$} \\
\hline
2     &1.59E+00&	0.87&	3.08E-01&	1.38 \\
4     &8.71E-01&	0.99&	1.18E-01&	1.82 \\
8     &4.37E-01&	1.00&	3.34E-02&	1.95 \\
16     &2.18E-01&	1.00&	8.63E-03&	1.99 \\
32     &1.09E-01& 1.00 &		2.18E-03& 2.00 \\
64   &5.45E-02		   &  &5.45E-04  & \\
\hline 
\multicolumn{5}{|c|}{$k=3$} \\
\hline
2     &1.01E-01&	2.94&	5.38E-03&	3.93 \\
4     &1.32E-02&	3.00&	3.53E-04&	4.02 \\
8     &1.66E-03&	3.01&	2.18E-05&	4.02 \\
16     &2.06E-04&	3.00&	1.34E-06&	4.01 \\
32     &2.57E-05&	3.00&	8.32E-08&	4.00 \\
64   & 3.21E-06&		&5.19E-09& \\
\hline
\multicolumn{5}{|c|}{$k=5$} \\
\hline
2     &2.45E-03&	4.94&	8.82E-05&	5.89 \\
4     &7.94E-05&	5.00&	1.49E-06&	5.98 \\
8     &2.49E-06&	5.00&	2.36E-08&	6.00 \\
16     &7.77E-08&5.00  &		3.69E-10&6.01	 \\
32     &2.42E-09		 &	&5.71E-12	&	 \\
\hline
\end{tabular} 
\label{tab:case3} }
\end{table}

\hilit{Other possibilities exist besides the three cases investigated, so,}
as a caveat, we present observations of suboptimal
convergence in the case $(k_u,k_q,k_v) = (3,0,3)$.  The DPG method is
uniquely solvable in this case: This would follow from
\autoref{thm:inj} once we prove that $\hB_h$ is injective. If $\hB_h
\hat z_n =0$, then by definition~\eqref{eq:Bh}, $\hb(\hz_n,v)=0$ for
all $v \in P_3(\oh),$ so in particular,
\[
\hz_n \in P_0(\d\oh): \quad 
\hb(\hz_n,v)=0,
\quad \forall\; v \in P_2(\oh).
\]
This implies, by the already known unisolvency of Case~1 with $k=1$,
i.e., $(k_u,k_q,k_v) = (1,0,2)$, and \autoref{thm:inj}, that
$\hz_n=0$. Therefore, the method is well-defined for the
$(k_u,k_q,k_v) = (3,0,3)$ case. Yet, the theory we presented does
\hilit{not guarantee optimal convergence rates in this case.  The
  numerical results reported in \autoref{tab:poor} show that the
  practically observed convergence rates in $H^1(\om)$ and $L^2(\om)$
  are indeed suboptimal in this case.  In fact, we observe second
  order convergence in $L^2(\om)$ as in case 1 with $k=1$.  An error
  analysis that proceeds exactly like the error analysis of case~3
  will predict this suboptimal rate (the rate being limited by the
  order $k_q$ of $\hX_h$).  However, the practically observed
  $H^1(\om)$ rates are higher than what the same analysis would
  predict.}

\begin{table}
\centering\footnotesize{
\caption{Poor $H^1(\om)$ and $L^2(\om)$ convergence for the case $(k_u,k_q,k_v) = (k,k-3,k)$} 
\begin{tabular}{|l|cc|cc|} 
\hline
$n$    & $\|u-u_h\|_{H^1(\om)}$& rate & $\| u - u_h \|_{L^2(\om)}$ &  rate \\ [0.5ex] 
\hline 
\multicolumn{5}{|c|}{$k=3$} \\
\hline
2     &1.02E-01&	2.85&	6.68E-03&	2.50 \\
4     &1.42E-02&	2.70&	1.18E-03&	1.99 \\
8     &2.18E-03&	2.38&	3.14E-04&	1.96 \\
16     &4.21E-04&	2.13&	8.06E-05&	1.99 \\
32     &9.59E-05&	&	2.03E-05&	\\
\hline
\end{tabular} 
\label{tab:poor}}
\end{table} 




\bigskip

\subsection*{Acknowledgements}

The authors are grateful to Leszek Demkowicz for discussions on the
subject and for the interaction opportunities provided in the
``ICES/USACM Workshop on Minimum Residual and Least Squares Finite
Element Methods'' (2013) where many questions such as those
addressed in this paper were formulated. Timaeus Bouma gratefully
acknowledges guidance from Tzanio Kolev during an internship at
Lawrence Livermore National Laboratory, where the issue of reducing
the degree of DPG test spaces was identified as practically relevant.


\end{document}